\documentclass[10pt]{amsart}
\usepackage[onehalfspacing]{setspace}
\usepackage{graphicx}
\usepackage{verbatim}
\usepackage{amsmath, amsthm, amssymb, amsfonts}
\usepackage{tikz}
\usetikzlibrary{decorations.markings}
\usepackage{enumerate}
\usepackage{mathtools}
\usepackage[T1]{fontenc}
\usepackage[utf8]{inputenc}
\usepackage{lmodern}
\usepackage[english]{babel}
\usepackage[autostyle]{csquotes}
\usepackage{fullpage}
\usepackage{tikz-cd}

\title{Self-Gluing formula of the monopole invariant and its application.}
\author{Gahye Jeong}

\begin{document}
\maketitle
\begin{abstract}
 Given a $4$-manifold $\hat{M}$ and two homeomorphic surfaces $\Sigma_1, \Sigma_2$ smoothly embedded in $\hat{M}$ with genus more than 1, we remove the neighborhoods of the surfaces and obtain a new $4$-manifold $M$ from gluing two boundaries $S^1 \times \Sigma_1$ and $S^1 \times \Sigma_1.$ In this artice, we prove a gluing formula which describes the relation of the Seiberg-Witten invariants of $M$ and $\hat{M}.$ Moreover, we show the application of the formula on the existence condition of the symplectic structure on a family of $4$-manfolds under some conditions. 
\end{abstract}

\tableofcontents
\newtheorem{thm}{Theorem}[section]
\newtheorem{df}[thm]{Definition}
\newtheorem{eg}[thm]{Example}
\newtheorem{prop}[thm]{Proposition}
\newtheorem{cor}[thm]{Corollary}
\newtheorem{rmk}[thm]{Remark}
\newtheorem{lem}[thm]{Lemma}
\numberwithin{equation}{section}
\numberwithin{figure}{section}
\newcommand{\tsigma}{\tilde{\Sigma}}

\section{Introduction}

 The Seiberg-Witten monopole invariant for $4$-dimensional manifolds was introduced by Witten in 1994 and has been actively studied in various aspects. It has been interesting to show how the invariants of different $4$-manifolds are related and how to compute the Seiberg-Witten monopole invariant from such relations. 

  Imagine that we decompose a $4$-manifold $M$ along a certain $3$-manifold $N$ given by the product of the circle with the surface and fill the obtained boundary naturally. We are interested in how to compute the Seiberg-Witten invariant for $M$ in terms of the resulting manifold $\hat{M}$. There exist such formulas for some cases. In \cite{km1}, it is proved in the case that the genus of the surface is more than one and the given $3$-manifold $M$ is separating in $X$. For the case of torus, the formula is proved in \cite{taubes} by a slightly different argument. Note that this is because that the Seiberg-Witten moduli space of 3-dimensional torus is different from the product of the circle and surfaces with genus more than $1$. In \cite{taubes}, he verifies not only  when $M$ is separating but also when $M$ is non-separating inside $X.$ This article mainly presents a statement and a proof of the case when $M$, that is the product of the circle with a surface with genus more than $1$, is separating in $X.$
 
  \begin{thm}[Main Theorem]\label{thm21} Let $M, \hat{M}$ be two closed and oriented four manifolds described previously. Detailed explanation on $M, \hat{M}$ is appeared in Subsection \ref{basicsetting}. Suppose that $b^{+}_2 (M) \ge 1$. It follows that $b^{+}_2 ( \hat{M}) \ge 1$. Suppose that $k \in H^2 (M, \mathbb{Z})$ is a characteristic cohomology class satisfying $k|_{N} = \rho^{*} k_0$ where $k_0 \in H^2 (C; \mathbb{Z})$ satisfies $$\langle k_0, [C] \rangle  = 2g (C) -2$$ and a natural projection $\rho : N \longrightarrow C.$  Let $\mathcal{K} (k)$ be a set of all characteristic classes $k' \in H^2 (M; \mathbb{Z})$ satisfying that the restrictions of $k$ and $k'$ on $\bar{M}$ are isomorphic and $k^2 = k'^2$. Similarly, for $\hat{k} \in H^2 (\hat{M}; \mathbb{Z})$ which is a characteristic cohomology class, let $\mathcal{\hat{K}} (\hat{k})$ be a set of all characteristic classes $\hat{k}' \in H^2 ( \hat{M}; \mathbb{Z})$ with the property that the restrictions of $\hat{k}$ and $\hat{k'}$ on $\bar{M}$ are isomorphic.  Then we have the following formula: $$\sum_{k' \in \mathcal{K}(k)} SW_M (k') = \sum_{\hat{k}' \in \mathcal{\hat{K}}(\hat{k}), \\ \hat{k}'^2 = k^2 - (8g-8)} SW_{\hat{M}} (\hat{k}').$$
 \end{thm}
 
 Moreover, we express the Seiberg-Witten monopole invariant of $4$-manifolds obtained from gluing two $4$-manifolds along multiple number of certain kinds of $3$-manifolds. It is from the direct application of the product formula in \cite{km1} and Main Theorem \ref{thm21}.
 
 Furthermore, this article presents its application on the existence of the symplectic structure on a certain type of $4$-manifolds in Section 3. This application is motivated by \cite{ni}. In \cite{ni}, the author investigated that for which knot $4$-manifold obtained from Fintushel-Stern surgery has a symplectic structure. Focusing on the fact that Fintushel-Stern surgery can be interpreted into symplectic normal sum of two $4$-manifolds, we will consider a normal connected sum of two simple types of four manifolds,
\begin{itemize}
\item $S^1 \times Y$ for three manifolds $Y$ 
\item $X$ which is a surface bundle over surface with the fiber having genus more than 2. 
\end{itemize} 
\begin{thm} Suppose that $Y$ is an oriented and closed three manifold with a submanifold $\Sigma \subset Y$ homeomorphic to the fiber of $X$ and  $b_1 (Y) = 1.$ Let $X^Y$ be the normal connected sum of $S^1 \times Y$ and $X$ along a surface $\Sigma \subset Y$ and the fiber of $X.$ Then,
$X^Y$ has a symplectic form $\omega$ and its canonical structure $K_{\omega}$ such that $\langle K_{\omega}, [\Sigma] \rangle = 2 g(\Sigma) - 2 $ if and only if Y is a surface bundle over the circle.
\end{thm}

\section*{Acknowledgement}
 The author appreciates her advisor, Professor Yi Ni for his suggestion of the topic of the paper and for essential ideas. She thanks to Nick Salter for suggesting the proof of Lemma \ref{lem1}. She also thanks to Siqi He to explain the standard gluing and limiting argument of the Seiberg-Witten theory.

\section{Self-gluing formula of the Seiberg-Witten invariant}

\subsection{The Seiberg-Witten equation and the invariant.}
 In this subsection, we will briefly review various types of the Seiberg-Witten equations and the invariants for $4$-manifolds. We will follow \cite{jm}, \cite[Section 2]{km1} and \cite{hutchings}. Let $X$ be a smooth, oriented Riemannian $4$-manifold with $Spin^c$-structure and let $g$ be a Riemannian metric on $X$. Let $\tilde{P}$ be a $Spin^c$-strucure on $(X, g)$. There is a corresponding vector bundle $S(\tilde{P})$ over X, that is induced from the representation of $Spin^c (4)$. $S(\tilde{P})$ is decomposed into $S^{+} (\tilde{P})$ and $S^{-} (\tilde{P})$ from the decomposition of the representation. We call these vector bundles {\bf spinor bundles} and sections of these vector bundles {\bf spinors}. 

 For a unitary connection $A$ on the determinant line bundle of $\tilde{P}$ and a section $\psi$ of $S^{+} ( \tilde{P})$, the Seiberg-Witten equations associated to a $Spin^c$-structure $\tilde{P}$ are defined by ${\bf SW}$: \begin{center}
 $F^{+}_A = q (\psi)$ \\$ D_A (\psi) =0,$\end{center} where $q$ is a natural quadratic form from $S^{+} (\tilde{P})$ to $\Lambda^2_{+} ( X ; i \mathbb{R})$ and $D_A$ is the usual Dirac operator using the Levi-civita connection on the frame bundle of $X$ and the connection $A$. When $X$ is closed, the index of the elliptic system defined from the equations $SW$ is given by $$d(\tilde{P}) = \frac{c_1 (\mathcal{L})^2 - (2 \chi (X) + 3 \sigma(X))}{4}$$ from the Atiyah-Singer index theorem. We define the moduli space $\mathcal{M}$ by the quotient of the space of solutions to $SW$ modulo the actions of gauge transformations. The moduli space itself is compact, however to make the moduli space smooth, we need a perturbation by a purely imaginary self-dual two-form.
 
  Specifically, for a generic $C^{\infty}$ self-dual real two-form $h$ on $X$, we define the perturbed Seiberg-Witten equations {$\mathbf{SW_h:}$}
\begin{center} $F_A^{+}=q(\psi) + ih$ \\ $D_A ( \phi) = 0.$ \end{center} The solution set $\mathfrak{m}$ has $C^{\infty} (X, S^1)$ action on itself, where $C^{\infty} (X, S^1)$ is a set of smooth functions from $X$ to $S^1.$ It is defined for $g \in C^{\infty} (X, S^1),$
\begin{center}
 $g(A,\psi) = (A - 2 g^{-1} dg, g\psi),$
\end{center} which is called {\bf gauge transformation}. Let the moduli space $\mathcal{M} (\tilde{P}, h)$ be a quotient of the solution space by $C^{\infty} (X, S^1).$ We call $\mathcal{M}(\tilde{P}, h)$ the moduli space of the $SW_h$ equation. Moreover, we define $\mathcal{M}^{0} (\tilde{P}, h) $ to be a quotient of the solution space by $\left\{ \phi \in C^{\infty} (X,S^1) : \phi(*) = 1 \right\}$, where $* \in X$ is a fixed base-point. When $\mathcal{A}$ is a set of unitary connections of the determinant line bundle of $\tilde{P}$, we call $$\mathcal{B}^{*}(\tilde{P}) := (\mathcal{A} \times S^{+} (\tilde{P})) / \text{gauge transformation}$$ {\bf configuration space of $\tilde{P}$}.

 We define $b_2^{+}(X)$ to be the dimension of any maximal subspace of the second cohomology $H^2 (X, \mathbb{R})$ on which the intersection form is positive definite. When $b_2^{+} >0,$ for generic $h$, the moduli space $\mathcal{M} (\tilde{P}, h)$ is smooth and  $\mathcal{M}^{0} (\tilde{P}, h)$ has a principal circle bundle structure over the moduli space $\mathcal{M} (\tilde{P}, h)$ since we expect no reducibles in the solution space.
 
 Now we will define the Seiberg-Witten invariant from this moduli space. We need to orient the moduli space. To orient the moduli space is equivalent to orient $H^0 (X, \mathbb{R}), H^1 (X, \mathbb{R}), H^2_{+} ( X, \mathbb{R})$. With a properly fixed orientation on the moduli space, we have a principal circle bundle $\mathcal{M}^{0} (\tilde{P}, h) $ over the moduli space $\mathcal{M} (\tilde{P}, h).$ Let $c\in H^2(\mathcal{M} (\tilde{P}, h))$ be the corresponding Chern class of this circle bundle. If the dimension of the moduli space is $2l$, which is even, then we define the Seiberg-Witten invariant to be  $\displaystyle \int_{\mathcal{M} (\tilde{P}, h)} c^l$. If the dimension is odd, then we define the invariant to be $0$.  Moreover, if $b^{+}_2 (X) >1$, then it is independent from the choice of the metric, i.e. an invariant of $X$. However, if $b^{+}_2 (X) =1$, the invariant depends on both the manifold $X$ and the metric.
 
 Therefore, when $b^{+}_2 (X) >1,$ this gives a function $$SW : \left\{Spin^c \text{-structures on $X$} \right\} \longrightarrow \mathbb{Z}.$$ It is often to use a function on the characteristic classes which amalgamates the information of $SW.$ Let $C(X) \in H^2 (X, \mathbb{Z})$ be the subset of characteristic cohomology classes i.e. its mod two reduction is equal to the second Stiefel-Whitney class. We have  $SW_X : C(X) \to \mathbb{Z}$ which is defined by $$SW_X (k) = \sum_{\substack{\text{$Spin^c$-structure $s$}\\ c_1(s) = k}} SW (s).$$

 Moreover, we will introduce the Seiberg-Witten invariant for $3$-manifold. Let $N$ be a $3$-dimensional Riemmanian manifold. Let $\tilde{P_N} \to N$ be a $Spin^c$-structure on $N.$ There is an associated irreducible complex spin bundle $S(\tilde{P_N})$ unique up to isomorphism. The $3$-dimensional Seiberg-Witten equations for $\tilde{P_N} \to N$ are given by $\mathbf{SW^3}:$
\begin{center}
$F_A = q(\psi)$ \\
$D_A (\psi) = 0,$
\end{center}
where $A$ is a unitary connection on the determinant line bundle of $\tilde{P_N}$ and $\psi$ is a section of complex spin bundle $S(\tilde{P_N}).$ Likewise, we have the perturbed Seiberg-Witten equation for $3$-manifolds. For any sufficiently small closed real two-form $h$ on $N$, we define the perturbed Seiberg-Witten equations $\mathbf{SW^3_h:}$
\begin{center} $F_A=q(\psi) + ih$ \\ $D_A ( \phi) = 0.$ \end{center} Similarly, we can define the configuration space and the moduli space of $SW^3_h$. Henceforth, let $\mathcal{C} (\tilde{P_N}), \mathcal{G} ( \tilde{P_N})$ and $\mathcal{B}^{*} (\tilde{P_N})$ denote $\mathcal{A} \times S^{+} (\tilde{P_N})$, the gauge transformation group and $\mathcal{C} (\tilde{P_N}) / \mathcal{G} ( \tilde{P_N})$ respectively.
\subsection{Self-Gluing formula}
\subsubsection{Basic setting}\label{basicsetting}
 Let $M$ be a closed oriented four-manifold and let $N = S^1 \times C \subset M,$ where $C$ is an oriented surface with genus $g >1$. When $N$ is a separating submanifold inside $M$, let $X \sqcup Y = M \setminus N.$ We fill $X, Y$ by gluing $D^2 \times C$ naturally along the boundaries $S^1 \times C$. $\hat{X}, \hat{Y}$ denote the resulting manifolds. Morgan, Szabo and Taubes showed the product formula on the Seiberg-Witten invariants in \cite{km1}. In other words, the Seiberg-Witten invariants of $\hat{X}$ and $\hat{Y}$ with properly fixed $Spin^c$-structures determine the Seiberg-Witten invariant of $M$.

 In this note, we want to investigate the case when $N$ is non-separating inside $M$. We consider a neighborhood $ N$, $nbd(N) = S^1 \times C \times [0,1]$ inside $M$. Let $\bar{M} = M \setminus nbd(N)$. Hence $\bar{M}$ is a smooth $4$-manifold with two boundary components $N \times \left\{ 0 \right\}$ and $N \times \left\{ 1 \right\}$. We fill the two boundaries of $\bar{M}$ by a trivial map $$D^2 \times C \hookrightarrow \bar{M}.$$ Let the resulting manifold be $\hat{M}$. 
  
 Conversely, we obtain $M$ from $\hat{M}$ by the following operation. When there are two $4$-manifolds with boundary $D^2 \times C_1, D^2 \times C_2$ smoothly embedded inside $\hat{M}$. Let $$\bar{M} =\hat{M}\setminus (D^2 \times C_1 \sqcup D^2 \times C_2).$$ After that, we glue the two boundaries of $\bar{M}$ by a map \begin{center} $S^1 \times C_1 \longrightarrow S^1 \times C_2$\\ $(x,y) \longrightarrow (\bar{x}, f(y)).$\end{center} $\bar{x}$ denotes a complex conjugate of $x$ and $f : C_1 \longrightarrow C_2$ is a diffeomorphism. Consequently, the resulting manifold becomes the original manifold $M$. In this section, we will verify the relation between the Seiberg-Witten invariants of $M$ and $\hat{M}$.
 
\subsubsection{The statement of the self-gluing formula}
 
  We introduce the main theorem of this section.
 
 \begin{thm}[Main Theorem] Let $M, \hat{M}$ be two closed and oriented four manifolds related as described previously in Subsection \ref{basicsetting}. Suppose that $b^{+}_2 (M) \ge 1$. It follows that $b^{+}_2 ( \hat{M}) \ge 1$. Suppose that $k \in H^2 (M, \mathbb{Z})$ is a characteristic cohomology class satisfying $k|_{N} = \rho^{*} k_0$ where $k_0 \in H^2 (C; \mathbb{Z})$ satisfies $$\langle k_0, [C] \rangle  = 2g -2$$ and a natural projection $\rho : N \longrightarrow C.$  Let $\mathcal{K} (k)$ be a set of all characteristic classes $k' \in H^2 (M; \mathbb{Z})$ satisfying that the restrictions of $k$ and $k'$ on $\bar{M}$ are isomorphic and $k^2 = k'^2$. Similarly, for $\hat{k} \in H^2 (\hat{M}; \mathbb{Z})$ which is a characteristic cohomology class, let $\mathcal{\hat{K}} (\hat{k})$ be a set of all characteristic classes $\hat{k}' \in H^2 ( \hat{M}; \mathbb{Z})$ with the property that the restrictions of $\hat{k}$ and $\hat{k'}$ on $\bar{M}$ are isomorphic.  Then we have the following formula: $$\sum_{k' \in \mathcal{K}(k)} SW_M (k') = \sum_{\substack{\hat{k}' \in \mathcal{\hat{K}}(\hat{k}), \\ \hat{k}'^2 = k^2 - (8g-8)}} SW_{\hat{M}} (\hat{k}').$$
 \end{thm}
 
 We will call this formula {\bf self-gluing formula}. Prior to the proof of the self-gluing formula, we will briefly review the product formula \cite[Theorem 3.1]{km1} and a sketch of its proof which  gives the idea.
\begin{thm}\label{productformula}\cite[Theorem 3.1]{km1} Supppose that $b^{+}_2 (\hat{X}), b^{+}_2 (\hat{Y}) \ge 1.$ It follows that $b^{+}_2 (M) \ge 1.$ Suppose that $k \in H^2(M; \mathbb{Z})$ is a characteristic cohomology class satisfying $k|_{N} = \rho^{*} k_0$ where $k_0 \in H^2(C; \mathbb{Z})$ satisfies $$\langle k_0, [C] \rangle = 2g-2.$$ Let $k_X$ and $k_Y$ be the restrictions of $k$ to $X$ and $Y.$ Consider the set $\mathcal{K} (k)$ of all characteristic classes $k' \in H^2 (M; \mathbb{Z})$ with the property that $k'|_X = k_X, k'|_Y = k_Y$ and $k'^2 = k^2$. We define $\mathcal{K}_X (k)$ to be all $l \in H^2 (X; \mathbb{Z})$ which are characteristic and satisfy $l|_X = k_X.$ The set $\mathcal{K}_Y (k)$ is defined analogously. Then for appropriate choices of orientations of $H^1 (M), H^1(X), H^1(Y)$ and $H^2_{+} (M), H^2_{+} (X), H^2_{+} (Y)$ determining the signs of the Seiberg-Witten invariants, we have 
\begin{equation}
\sum_{k' \in \mathcal{K}(k)} SW_M (k') = (-1)^{b(M,N)} \sum SW_{\hat{X}} (l_1) SW_{\hat{Y}} (l_2)
\end{equation}
where $b(M,N) = b_1 (X, N) b^{\ge}_2 (Y, N),$ and the sum on the right-hand-side extends over all pairs $(l_1, l_2) \in \mathcal{K}_X (k) \times \mathcal{K}_Y (k)$ with the property that $$l_1^2 + l_2^2 = k^2 - (8g - 8).$$ 
\end{thm}

 We summarize a sketch of the proof of Theorem \ref{productformula}. First, they defined the moduli space of the perturbed Seiberg-Witten equations for one cylindrical end $4$ manifolds $X$ and $Y$. They showed the moduli spaces are smooth and compact with the correct dimension, which is given by the index theorem. Next, by gluing together two configuration spaces of $X$ and $Y$, they showed the union of the product of the two moduli spaces with fixed Chern integral is diffeomorphic to the moduli space of the original manifold $M$. Third, they defined the relative Seiberg-Witten invariants for one cylindrical end $4$ manifolds. The above diffeomorphism between moduli spaces implies that the Seiberg-Witten invariant of $M$ can be represented by the relative Seiberg-Witten invariants of $X$ and $Y$. Lastly, the relative Seiberg-Witten invariants of $X$ and $Y$ are equal to the Seiberg-Witten invariants of $\hat{X}$ and $\hat{Y}$ respectively, which implies the statement. 
\subsubsection{The outline of Section 2.}

 In Subsection \ref{sub23}, we define the moduli space of the perturbed Seiberg-Witten equation for $4$-manifolds with a certain type of cylindrical ends. We show this moduli space is compact, smooth and has a nice decaying property. Furthermore, we explain how the orientation of the moduli space is given. In Subsection \ref{sub25}, we define the relative Seiberg-Witten invariant for $4$-manifolds with cylindrical ends and show that the Seiberg-Witten invariant of the original closed $4$-manifold $M$ is associated with the relative Seiberg-Witten invariant of the resulting manifold $\bar{M}$ obtained from cutting along non-separating $3$-submanifold.  Lastly, we show that the relative Seiberg-Witten invariants of $\bar{M}$ is equal to the Seiberg-Witten invariant of $\hat{M}$ up to orientations in Subsection \ref{sub26} by generalizing the gluing theorem of the Seiberg-Witten moduli spaces for $4$-manifolds with boundary \cite[Theorem 9.1.]{km1}. In subsection \ref{sub32}, we state the gluing formula when we glue two $4$-manifolds along not only one connected surface, but multiple oriented surfaces.

\subsection{The moduli space of $4$-manifolds with cylindrical ends.}\label{sub23}
 We introduce the moduli space for $4$-manifolds with a certain type of cylindrical ends $[0, \infty) \times S^1 \times C,$ where $C$ is an oriented surface which is not necessarily connected. We will mainly follow the way to define this moduli space in \cite[Section 6]{km1} and \cite[Section 24]{km2}. Notice that most definitions and results in this section are exactly same with the statement in \cite[Section 6]{km1} except that we allow disconnected ends. 
 
 Suppose that $X$ is a $4$-manifold with cylindrical end $[0, \infty) \times S^1 \times C$ when $C$ is an oriented surface, which is not necessarily connected. Let $\displaystyle C = \bigsqcup_{\alpha} C_{\alpha}$, where $C_{\alpha}$ is a connected component of $C$. Suppose that the genus of $C_{\alpha}$ are all equal to $g$, which is at least $2$. Suppose that $N = S^1 \times C$. We fix a $Spin^c$-structure $\tilde{P}$ whose determinant line bundle restricted on $S^1 \times C_{\alpha}$ is isomorphic to the pullback of the line bundle on $C_{\alpha}$ whose degree is equal to $(2g-2)$ on $C_{\alpha}$. Before describing the moduli space of $X$, we briefly review the Seiberg-Witten equations of two manifolds, 
\begin{enumerate}
\item $S^1 \times C$ for a connected oriented surface $C$ \item cylindrical $4$-manifolds, which is diffeomorphic to $I \times N$ for a $3$-manifold $N.$
\end{enumerate}

\subsubsection{Preliminaries} Let $N = S^1 \times C.$ We consider a $Spin^c$-structure $\tilde{P_N}$, whose determinant line bundle $\mathcal{L}$ has a degree $\pm (2-2g)$ on every component of $C$. If the solution exists on the $Spin^c$-structure $\tilde{P_N}$, then $\tilde{P_N}$ is a pull-back $Spin^c$ structure from a $Spin^c$-structure on $C$ \cite[Proposition 5.1.]{km1}. 
\begin{prop}\cite[Corollary 5.3.]{km1}\label{3ma}
	For any sufficiently small closed real two-form $h$ on $N$, there is a unique solution to the perturbed Seiberg-Witten equations $(SW^3_h):$
\begin{center} $F_A=q(\psi) + ih$ \\ $D_A ( \psi) = 0.$ \end{center} This solution represents a smooth point of the moduli space in the sense that its Zariski tangent space is trivial. 
\end{prop}
 We remark that Proposition \ref{3ma} is originally proved only for a connected surface $C$ in \cite{km1}. However, the statement is also true when $C$ is disconnected since the solution restricted to each component satisfies Proposition \ref{3ma}.

 Next, we focus on the case in which a smooth and oriented Riemannian $4$ manifold $X$ is orientation-preserving isometric to $I \times N$ where $N$ is a closed oriented three manifold and $I$ is a (possibly infinite) open interval. Let $\tilde{P_N}$ be a $Spin^c$-structure on $N$. There is a natural $Spin^c$-structure $I \times \tilde{P_N}$ over $I \times N.$ With the $Spin^c$-structure, they introduced the perturbed monopole equations $\mathbf{SW_h}$ on $I \times N$ for $N=S^1 \times C$: \begin{center} $F_A=q(\psi) + ih$ \\ $D_A ( \psi) = 0.$ \end{center} where $\star$ is the complex-linear Hodge-star operator on $N$. Let $h$ be $\star n + dt \wedge n$ where $n$ is a harmonic one-form on $C$ and $dt$ denotes one-form on $I$-direction. 

  We fix a background $C^{\infty}$-connection $A_0$ over $N.$ Suppose that $f : C^{*} ( \tilde{P}_N) \longrightarrow \mathbb{R}$ is a functional introduced in \cite[Section 6.2.]{km1} by 
 $$f(A, \psi) = \int_{N} F_{A_0} \wedge a + \frac{1}{2} \int_{N} a \wedge \text{d}a + \int_N \langle \psi, D_A \psi \rangle \text{dvol}$$ where $a = A -A_0.$ Every $C^{\infty}$ solution on the Seiberg-Witten equation on $X$ corresponds to a path in the configuration space of $N$, $\gamma : I \to C ( \tilde{P_N}).$  Moreover, the Seiberg-Witten equation on $X$ can be viewed as a gradient flow equation of a map from $I$ to $\mathcal{C} (\tilde{P_N})$ in the following way. See \cite[Section 6]{km1} for details.
\begin{prop}\cite[Proposition 6.6.]{km1} Fix an open interval $I.$ If a configuration $(A(t), \psi(t))$ in a temporal gauge for the $Spin^c$-structure $I \times \tilde{P_N} \to I \times N$ satisfies the Seiberg-Witten equations, then it gives a $C^{\infty}$-path in $\mathcal{C}(\tilde{P_N}),$ $$\frac{\partial (A,\psi)}{\partial t} = \nabla f (A,\psi).$$ Two solutions to the Seiberg-Witten equations are gauge-equivalent if and only if the paths in $\mathcal{C} (\tilde{P_N})$ that they determine in temporal gauges are gauge-equivalent under the action of $\mathcal{G} (\tilde{P_N}).$ 
\end{prop}
  
 \begin{df}\label{finite} With the above settings, we call a $C^{\infty}$-solution whose associated flow line $\gamma : [0,\infty) \longrightarrow C^{*}(\tilde{P_N})$ satisfies $$\lim_{t \rightarrow \infty} {f(\gamma(t)) - f (\gamma(0))} < \infty$$ \textbf{a finite enery solution}.
\end{df}
 Remark that Definition \ref{finite} is an analogous definition of \cite[Definition 8.1]{km1}. Furthermore, the gradient flow equation of the cylinder exponentially decays when $N = S^1 \times C$, where $C$ is a connected, oriented surface with $g(C) >1.$

\begin{prop}\cite[Lemma 6.15.]{km1}\label{cyl}
	 With $N = S^1 \times C$ and $\tilde{P_N}$ a $Spin^c$-structure whose determinant line bundle $\mathcal{L}$ has degree $\pm (2-2g)$ on $C$, there are positive constants $\epsilon, \delta >0$ such that for any $T \ge 1$ if $(A(t), \psi(t))$ is a solution to the Seiberg-Witten equations on $[0,T] \times N$ in a temporal gauge and if for each t, $0 \le t \le T$, the equivalence class of $(A(t), \psi(t))$ is within $\epsilon$ in the $L^2_1$-topology on $B^{*} ( \tilde{P_N})$ of the solution $[A_0, \psi_0]$ of the Seiberg-Witten equations on $N$, then the distance from $[A(t), \psi(t)]$ to $[A_0, \psi_0]$ in the $L^2_1$-topology is at most $$d_0 \exp ( \delta t) + d_T \exp (-\delta (T-t)),$$ where $d_x$ is the $L^2_1$-distance from $[A(x), \psi(x)]$ to $[A_0, \psi_0]$ when $x = 0, T$.           
\end{prop}

\subsubsection{Definition of the moduli space.} 

 Now we are ready to define the moduli space for 4-manifolds with cylindrical ends. We fix a small perturbation of the monopole equations for $X$. Then, the equations are like the following
 $\mathbf{SW_{h_{X} + \mu^{+}_{X}}} :$ $$F^{+}_{A} = q(\psi) + i \phi_X( h ) + i\mu^{+}_{X}$$ $$D_A \psi = 0.$$
\begin{itemize}
\item $n$ harmonic one-form on $C$.
\item $h = \star n + dt \wedge n$.
\item $\mu^{+}_{X}$ is a compactly supported self-dual two form.
\item $\phi_X$ is a $C^{\infty}$ function which is identically $1$ on $[0, \infty) \times S^1 \times C$ and vanishes off of $[-1, \infty) \times S^1 \times C.$ 
\end{itemize}

	 We fix a $Spin^c$-structure $\tilde{P}$ on $X$. The restriction of $\tilde{P}$ to $N$ is denoted by $\tilde{P_N}.$ For each $C^{\infty}$ solution to the Seiberg-Witten equation with respect to $\tilde{P}$, there exists a temporal gauge for $\tilde{P}$ restricted to the cylindrical end. The restriction of the solution on the cylindrical end has an associated gradient flow line $\gamma$ from $[0, \infty)$ to $C^{*} ( \tilde{P}_N)$. We call a $C^{\infty}$-solution of $SW_{h_X + \mu^{+}_X}$ whose restriction on its cylindrical ends is a finite energy solution {\bf finite energy solution on cylindrical-end manifold $X$}. 
    
    Let $\tilde{\mathcal{M}} ( \tilde{P}, n, \mu^{+}_X)$ be the set of all finite energy solutions to the $SW_{h_X + \mu^{+}_X}.$ The topology on $\tilde{\mathcal{M}} ( \tilde{P}, n, \mu^{+}_X)$ is described in \cite[Section 8]{km1}. The space obtained from $\tilde{\mathcal{M}} ( \tilde{P}, n, \mu^{+}_X)$ quotient by the gauge transformation group which consists of $C^{\infty}$-change of gauge is denoted by the moduli space $\mathcal{M} ( \tilde{P}, n, \mu^{+}_X).$ The moduli space can be regarded as the zeroes of a map with Fredholm differential modulo the gauge transformation group.

\begin{df}
 For each solution $(A,\psi)$ of the Seiberg-Witten equation on $X$, we define $$ c(A, \psi) = - \frac{1}{4\pi^2} \int_X F_A \wedge F_A.$$ We call $c(A, \psi)$ the Chern integral of the solution. 
\end{df}

 The index of the Fredholm complex is equal to $$\frac{1}{4} [ c(A, \psi) - 2 \chi (X) - 3 \sigma (X)]$$ as with the of $SU(2)$  ASD connection in \cite{jm}.  The Chern integral defines a continuous map on $\mathcal{M} ( \tilde{P_X}, n, \mu^{+}_X.)$ Let $\mathcal{M}_{c} ( \tilde{P_X}, n, \mu^{+}_X)$ denote the pre-image of $c.$ From a choice of appropriate perturbation $\mu_X^{+}, n,$ the moduli space becomes a smooth manifold with the expected dimension. Moreover, the moduli space is compact from the same argument in \cite[Proposition 8.5.]{km1}.

\subsubsection{Exponential decaying property of the moduli spaces.}
 The other way to investigate the moduli spaces for the cylindrical $4$-manifolds is to show the analogous results in \cite[Section 8]{km1}. This is an analogous result of \cite[Corollary 8.6]{km1}. The statements is exactly same with \cite[Corollary 8.6]{km1}, except the fact that we allow that the cylindrical end may be disconnected. 
 \begin{prop}\label{unif}
 With the notations above, let $Y = S^1 \times C$ and $X$ be a complete Riemannian $4$-manifold with cylindrical-end isomorphic to $[-1, \infty ) \times Y$. Let $\tilde{P}$ be a $Spin^c$-structure whose restriction to $Y$ is isomorphic to the pull back from $C$ of a $Spin^c$ structure whose determinant line bundle has degree $\pm (2-2g)$. Then for any $c_0$ the following holds for any sufficiently small harmonic one form $n \ne 0 \in \Omega^1 ( C; \mathbb{R})$ and every $c \ge c_0$: There is a constant $T \ge 1$ such that if $(A, \psi)$ is a finite energy solution to the equations $SW_h$ with Chern integral $c$, then for every $t \ge T$ the restriction $(A(t), \psi(t))$ is within $\exp ( - z (t-T))$ in the $L^2_1$-topology of a solution to the equations $SW_{\star n}$ on $Y$, where the constant $z$ depends only on $Y$. The same result holds when the curvature equation is replaced by $$F^{+}_A = q(\psi) + ih + i \mu^{+}$$ for any sufficiently small, compactly supported, self-dual two-form $\mu^{+}$ on $X$. 
 \end{prop}

\subsubsection{Orientations of the moduli space.}   We review the way to orient the moduli space for $4$-manifolds with cylindrical end based
on \cite[Section 9.1]{km1}, \cite[Chapter 24.2]{km2}. Let $Z$ be a $4$-manifold with cylindrical end and let $T$ be a cylindrical neighborhood of infinity in $Z$. We define $H^2_{\ge} (Z,T; \mathbb{R})$ to be the maximal subspace of the second cohomology group $H^2_{\ge} (Z, T; \mathbb{R})$ whose intersection pairing is positive semi-definite. With the notations above, we have the following proposition.

\begin{prop}\cite[Corollary 9.2.]{km1} To orient the moduli space of finite energy solutions to Seiberg-Witten equations on a cylindrical-end manifold $Z$, it suffices to orient $H^1 (Z, T ; \mathbb{R}) \oplus H^2_{\ge} (Z, T; \mathbb{R})$.
\end{prop}
\begin{rmk}\label{rmk1}We introduce another way to orient the closed manifold $M$ containing the cylinder $T = N \times \mathbb{R}$ as a submanifold, where $N$ is a compact $3$-manifold inside $M$. We can orient the moduli space for $M$ by orienting $H^1 (M, T; \mathbb{R}) \oplus H^2_{\ge} (M, T; \mathbb{R}).$ This is from the same logic in \cite[p.772]{km1}. We will use this remark to trace the orientations when gluing the moduli spaces.  
\end{rmk}

\subsection{Relation between two moduli spaces over $\bar{M}$ and $M$.}\label{sub25}

\subsubsection{Relative Seiberg invariant for manifolds with two cylindrical ends.}\label{sub24}

 Originally, the Seiberg-Witten invariant was described for closed manifolds, however it can be extended to $4$-manifolds with cylindrical ends which is denoted by the relative Seiberg-Witten invariants. On the assumption that the moduli space for such manifold is smooth and compact, the relative invariant is analogously defined as the original Seiberg-Witten invariant. We will follow Section 9.2 of \cite{km1}.
 \begin{df}\label{def1}
 	Let $M$ be an oriented, complete, Riemannian four manifold with cylindrical ends $T$ isometric to $[0, \infty) \times S^1 \times C$, where $C = C_1 \sqcup C_2$ and $C_1, C_2$ is an oriented, connected surface with $g(C_1) = g(C_2) > 1$. Let $\tilde{P_M}$ be a $Spin^c$-structures whose restriction on $S^1 \times C_i$ is isomorphic to the pullback of a $Spin^c$ structure on $C_i$ whose determinant line bundle is degree $(2g-2)$ for $i=1,2$. We defined the smooth and compact moduli space $\mathcal{M}_c (\tilde{P_M}, n, \mu^{+})$. This is the space of solutions of the Seiberg-Witten equation with a small perturbation $n, \mu^{+}$ whose Chern integral is equal to $c$. 
    Let $c_1$ be the first Chern class of the universal circle bundle over this moduli space. When the dimension of the moduli space is $2d$, we define the relative Seiberg-Witten invariant $SW_c (\tilde{P_X})$ by the pairing of ${c_1}^d$ and $ \mathcal{M}_c (\tilde{P_M}, n, \mu^{+})$. If the dimension is odd, then we define $SW_c (\tilde{P_X}) = 0.$ This relative Seiberg-Witten invariant is same for all $n$ and $\mu^{+}$.
 \end{df}
 
 We have a powerful tool to simplify the relative invariant. 
\begin{prop} 
	Let $M$ be a manifold and $\tilde{P_M}$ be a $Spin^c$-structure described in Definition \ref{def1}. Then for any c satisfying that the dimension of the moduli space $\mathcal{M}_c (\tilde{P_M}, n, \mu^{+})$ is equal to $2d >0$,  the value of the relative Seiberg-Witten invariant $SW_c (\tilde{P_M})$ is zero. 
\end{prop}
 The proof is identical to the proof of Proposition 9.4. in \cite{km1}. Even $C$ is not connected, the fact that every point in the moduli space is asymptotic at infinity to the same irreducible configuration on $S^1 \times C$ is still true. Henceforth, it is enough to consider the case that the dimension is equal to $0$ when we deal with the relative Seiberg-Witten invariant of such manifolds

 Suppose that $N = S^1 \times C$ is smoothly embedded inside $M$. Let $g$ be the chosen metric on $M$ so that the metric $g$ is isometric to the product metric on the neighborhood of $N$, $[-1,1] \times N$ inside $M.$ $M \setminus N$ is a closed, oriented four manifold with two cylindrical ends $[-1, \infty) \times N_1, [-1, \infty) \times N_2$, where $N_1$ and  $N_2$ are homeomorphic to $N$. Let $\bar{M}$ denote $M \setminus N$. Moreover, the Riemannian metric on $\bar{M}$ is induced from $(M,g).$
 
 For all $s \ge 1 $, let $M_s$ be the closed Riemmanian four manifold obtained by gluing $$\left\{ s \right\} \times N_1 \longrightarrow \left\{ s \right\} \times N_2$$ $$(z,w) \longrightarrow (\bar{z}, w)$$ after truncating $\bar{M}$ at $\left\{ s \right\} \times N_1, \left\{ s \right\} \times N_2$. On the other hand, we could interpret a family of manifolds $M_s$ parameterized by $s \in [1, \infty)$ as riemmanian manifolds $(M, g_s)$ given a family of metrics $\left\{g_s \right\}$ on $M.$ Let $\nu$ be the product neighborhood $[-1, 1] \times N$ inside $M.$ For all $s$, $g_s$ is identical to the original metric $g$ on $M$ outside $\nu.$ 
 
 On $\nu,$ $g_s|_{\nu} = \lambda_s (t)^2 \text{dt}^2 + \text{d}\theta^2 + \text{d}\sigma^2,$ where $\text{dt}^2$ is the usual metric on $[-1,1]$, $\text{d}\theta^2$ is the usual metric on $[0,2 \pi]$ on $S^1$ and $\text{d}\sigma^2$ is the fixed (constant curvature) metric on $C.$ $\lambda_s : [-1, 1] \to \mathbb{R}$ is a even function satisfying that $\lambda_s (t) =1$ for $|t| > \frac{1}{2}$ and $\int_{-\frac{1}{2}}^{\frac{1}{2}} \lambda_s (t) \text{dt} = s.$ Conclusively, we can think Riemannian manifold $M_s$ of $(M, g_s).$ 
 
  Let $N_{-}, N_{+}$ be submanifolds $\left\{-\frac{1}{2} \right\} \times N, \left\{\frac{1}{2} \right\} \times N$ in $M_s$ and $T_s$ be the cylinder inside $M_s$ bounded by $N_{-}, N_{+}.$ We add one remark that any two-form $\omega$ on $M$ can be extended to the two-form on $M_s$ in a obvious way so that on the restriction on cylindrical part $T_s$ the two-form is nonzero and constant.

 For all positive $e$, let $\mathcal{M}_{e} (\tilde{P}_{\bar{M}}, n, \mu^{+}_{\bar{M}})$ be the moduli space of finite energy solutions to the perturbed equations with Chern interal $e$. By choosing sufficiently small and generic $n$ and $\mu^{+}_{\bar{M}}$ generically, we can arrange that this is a smooth and compact moduli space.

  Let $S$ be the set of isomorphism classes of $Spin^c$ structures $\tilde{P}$ on $M$ with the property that $\tilde{P}|_{\bar{M}} \cong \tilde{P}_{\bar{M}}.$ $S_e$ denotes a subset of $S$, which consists of $Spin^c$ structures whose determinant line bundle $\mathcal{L}$ satisfies $c_1 (\mathcal{L})^2 = e$. For $\tilde{P} \in S_e$ with $e \ge 0$, there is the corresponding $Spin^c$ structures $\tilde{P}_s$ over $M_s$. For any $s$ sufficiently large, we denote by $\mu^{+}_s$ the self-dual form on $M_s$, which is equal to $\mu^{+}_{\bar{M}}$. For large $s$, $supp (\mu^{+}_{\bar{M}}) \subset M_s.$
  
 We define $\mathcal{M} ( \tilde{P}_s, h_s, \mu^{+})$ to be the moduli space of solutions to the perturbed SW equations $SW_{h_s + \mu^{+}} :$ $$F^{+}_A = q(\psi) + i \phi_s ( \star n + dt \wedge n ) + i \mu^{+}.$$ $$D_A ( \psi) = 0.$$
 \begin{itemize}
 \item $\phi_s : M_s \longrightarrow [0,1],$ which is defined similarly as $\phi_X$ in $SW_{h_X + \mu^{+}_X}$.
 \item $h_s = \phi_s ( \star n + dt \wedge n).$
 \end{itemize}

\begin{thm} With the notations and assumptions above, suppose that $n$ is sufficiently small and generic and $s$ sufficiently large. There is a diffeomorphism 
\begin{equation}\label{bijection}
\mathcal{M}_e ( \tilde{P}_{\bar{M}}, n, \mu^{+}_{\bar{M}}) \xrightarrow{\Phi} \bigsqcup_{\tilde{P} \in S_e} \mathcal{M} ( \tilde{P_s}, n , \mu^{+}),
\end{equation}
determined by gluing the two boundary parts of the solution and deforming slightly so as to be in the solution moduli space. 
\end{thm}
 The proof follows the standard gluing arguments and limiting arguments. The original description for the arguments is from \cite[Chapter 7]{donaldson}. We will briefly explain three main ingredients to apply the standard gluing arguments. 
 
 First, the moduli space of solutions over $S^1 \times C$, $g(C) >1,$ consists of a smooth single point. In other words, all of the solutions in the configuration space of $S^1 \times C$ are gauge-equivalent. For an element $[A, \phi]$ in $\displaystyle \mathcal{M}_e ( \tilde{P}_{\bar{M}}, n, \mu^{+}_{\bar{M}})$, after restricting $[A, \phi]$ on $M_s$, we can glue the two boundary parts of $[A, \phi]$ and deform slightly so as to be an element in $\displaystyle \bigsqcup_{\tilde{P} \in S_e} \mathcal{M} ( \tilde{P_s}, n , \mu^{+}).$ This defines a map $\Phi.$ Whatever the resulting pair obtained from gluing two solutions is, we can find the solution in a neighborhood of the pair in the configurations space. Moreover, if we slightly deform the element in the configuration element, then the element does not change in the moduli space. These are from the same argument in \cite[Section 7.2.]{donaldson}. 
 
a Second, both $\mathcal{M}_e ( \tilde{P}_{\bar{M}}, n, \mu^{+}_{\bar{M}})$ and $\displaystyle \bigsqcup_{\tilde{P} \in S_e} \mathcal{M} ( \tilde{P_s}, n , \mu^{+})$ are smooth and compact as proved in Section \ref{sub23}. Moreover, on cylindrical ends, we have the uniform decay results from Proposition \ref{unif}. This implies the injectivity of $\Phi.$
 
 Lastly, we have the exponentially decaying result of $T_s$ over the manifold $M_s$, which is uniform with respect to $s:$
 
  \begin{prop}\cite[Corollary 7.5.]{km1}\label{sur} There is a constant $K>0$ depending only on $M$ and $\tilde{P}$ such that for any harmonic one-form $n \ne 0 \in \Omega^1 (C ; \mathbb{R})$ sufficiently small and for any $s \ge 1 $ and any solution $(A, \psi)$ to the perturbed Seiberg-Witten equations $SW_{h_s}$ on $M_s$, the restriction of $(A, \psi)$ satisfies the following. 
 For any $t \in [0,s]$, we have that the $L^2_1$ distance from $(A(t), \psi(t))$ to a static solution is at most $K \exp ( - \delta d (t))$ where $\delta$ is the constant in Proposition \ref{cyl} and $d(t) = \min (t,s-t)$.
  \end{prop} This property replaces \cite[Proposition (7.3.3)]{donaldson} which is an essential fact for the surjectivity. There is a formula between the relative invariant of $\bar{M}$ and the original invariant of $M$, followed by Proposition \ref{sur}.
  
\begin{thm}\label{thm211} With the notations defined in the above section, let $S_e$ be a set of $Spin^c$-structures $P$ on $M$ such that the restriction on $\bar{M}$ is isomorphic to $\tilde{P}_{\bar{M}}$ and its determinant line bundle $\mathcal{L}$ satisfies that $c_1 (\mathcal{L} )^2 = e$. By orienting $H^1 ( \bar{M}, T) \oplus H^2_{\ge} ( \bar{M}, T)$, we fix the sign for the relative Seiberg-Witten invariant of $\bar{M}$. Moreover, this determines the orientation $H^1 (M) \oplus H^2_{\ge} (M)$ which fix the sign of Seiberg-Witten invariant of $M$. With these fixed orientations, we have the following formula:
$$\sum_{\tilde{P} \in S_e} SW(\tilde{P}) = SW_e ( \tilde{P}_{\bar{M}}).$$
\end{thm}

\subsection{The Generalized Gluing Theorem for moduli spaces}\label{sub26}
	 In \cite[Theorem 9.1.]{km1}, the authors explained how to glue two configuration spaces for $4$-manifold with the specific cylindrical end. To summarize, the gluing theorem says the following: if we glue two one-cylindrical-end 4-manifolds, then the moduli space of the resulting manifold is represented by the product of the moduli spaces of the two original manifolds. We can naturally generalize the gluing theorem to the case that one of the two original manifolds have two cylindrical ends.
     
     Let $X$ denote a $4$-manifold with two cylindrical ends and let $Y$ denote a $4$-manifold with a cylindrical end. Y has a cylindrical end $T = [0,\infty) \times N$ where $N = S^1 \times C$ and $C$ is an oriented, connected surface with genus $g > 1 $. Let $C_1$ and $C_2$ be homeomorphic to $C$. $X$ has two cylindrical ends $T_i = [0, \infty) \times N_i$ where $N_i =  S^1 \times C_i ( i=1,2)$. We introduce two following notations: $T^s = [0,s] \times N$ and $T^s_i = [0,s] \times N_i.$ We fix $Spin^c$-structures $\tilde{P_X}$ and $\tilde{P_Y}$ whose determinant line bundles restricted to $N, N_1, N_2$ are all isomorphic to the pull back from $C, C_1, C_2$ of a line bundle of degree $(2g-2)$ on $C, C_1, C_2$ respectively. As we did in Section \ref{sub25}, we truncate $X, Y$ at $N_2 \times \left\{s \right\}$ and $N \times \left\{ s \right\}$. By gluing along $N_2 \times \left\{ s \right\} \subset X$ and $N \times \left\{ s \right\}  \subset Y$, we obtain a new cylindrical-end $4$-manifold $M_s$. Let $T'^s$ denote the cylinder $T^s_2 \cup T^s \subset M$. Let $S$ be a set of $Spin^c$-structures $\tilde{P}$ on $M_s$ such that $\tilde{P} |_X = \tilde{P_X}|_{X_s}$ and $\tilde{P}|_Y = \tilde{P_Y}|_{Y_s}$. We have the following diffeomorphism between moduli spaces:
\begin{equation}\label{gluing} \bigsqcup_{c_1 + c_2 = e} \mathcal{M}_{c_1} ( \tilde{P_X}, n, \mu_X^{+}) \times \mathcal{M}_{c_2} ( \tilde{P_Y}, n, \mu_Y^{+}) \xrightarrow{\cong} \bigsqcup_{\tilde{P} \in S} \mathcal{M}_c ( \tilde{P}, n, \mu^{+}).\end{equation} The same argument in \cite{km1} can be applied to show that the gluing map induces diffeomorphism. Specifically, the following three facts supporting the argument are true in this case.  First, the moduli space of $S^1 \times C$ is formed of a single smooth point. Second, the moduli spaces appeared in Equation \ref{gluing} are compact and smooth with the exponential uniform decay. Lastly, the solutions on the center tube $T'^s$ in $M_s$ decay uniformly in $s$. See the last paragraph in \cite[p.770]{km1} for more details.

\begin{thm}[Generalized Gluing Theorem]\label{thm2} We follow the notations $S, M_s, X, Y, N_1, N_2, N$ defined above. By orienting $H^1 ( X, T_1 \cup T_2; \mathbb{R}) \oplus H^2_{\ge} (X, T_1 \cup T_2; \mathbb{R})$ and $H^1 ( Y, T; \mathbb{R}) \oplus H^2_{\ge} (Y, T; \mathbb{R})$, we can orient the moduli spaces appeared on the left hand side. With these choices of the orientations, we can determine the orientation of the moduli spaces on the right hand side by orienting $H^1 ( M_s, T'^s; \mathbb{R}) \oplus H^2_{\ge} (M_s, T'^s; \mathbb{R})$ With the choices of orientations, we have the following product formula:
$$\sum_{\tilde{P} \in S} SW_c (\tilde{P}) = (-1)^{b^1 (X,T_1 \cup T_2) b^2_{\ge} (Y,T)} \sum_{c_1 + c_2 = c} SW_{c_1} (\tilde{P_X}) SW_{c_2} ( \tilde{P_Y}).$$
\end{thm}
 The orientation term $(-1)^{b^1 (X,T_1 \cup T_2) b^2_{\ge} (Y,T)}$  in Theorem \ref{thm2} comes from Remark \ref{rmk1} and \cite[Section 9.1.]{km1}.

 In \cite[Section 9.4]{km1}, the relative invariant of $D^2 \times C$ is computed, where $C$ is a connected and oriented genus $g>1$ surface. The invariant $SW_c ( \tilde{P})$ is zero, unless $c = 4-4g$. In the case of $c= 4-4g$, the relative invariant is equal to $1$ with the proper orientation which orients the moduli space positively. Thus, the following corollary comes from assigning one manifold to $D^2 \times C$ in \cite[Theorem 9.1.]{km1}. 
 
\begin{prop}\cite[Corollary 9.9]{km1} Let $X$ be an oriented Riemannian four-manifold with a cylindrical end isometric $[0,\infty) \times S^1 \times C$, where $C$ is a connected and oriented surface with genus $g>1$. Let $\hat{X}$ be the closed four manifold obtained by filling in $X$ with $D^2 \times C$. Then for $Spin^c$-structure $\tilde{P} \longrightarrow \hat{X}$ satisfying that the determinant line bundle $\mathcal{L}$ of $\tilde{P}$ has degree $(2g-2)$ on $\left\{ 0 \right\} \times C$, we have $$ SW(\tilde{P}) = SW_c ( \tilde{P}|_X)$$ where $$c + (4-4g) = \langle c_1 (\mathcal{L})^2, [\hat{X}] \rangle. $$
\end{prop}
 
 Schematically speaking, if there is an $D^2 \times C$ embedded inside the four-manifold, then the relative invariant of the manifold obtained by removing $D^2 \times C$ is equivalent to the invariant of the original manifold. Thereafter we want to prove that even if we remove another $D^2 \times C$ inside the resulting manifold, the relative invariant still remains unchanged. The following corollary comes after the Generalized Gluing Theorem \ref{thm2}. 

\begin{prop}\label{prop9} Let $X$ be a compact and oriented $4$ manifold with two cylindrical ends $T_i = [0, \infty) \times N_i$, where $N_i = S^1 \times C_i$ and $i=1,2$. Let $g(C_1) = g(C_2) = g >1$. We can fill $\left\{\infty\right\} \times N_i$ by gluing $D^2 \times C_i$. Then let $\hat{X}$ be the manifold obtained $X$ by filling two cylindrical ends with $D^2 \times C_i$ for $i=1,2$. Then, for $Spin^c$-structures $\tilde{P} \longrightarrow \hat{X}$ whose determinant line bundle $\mathcal{L}$ restricted on $\left\{0\right\} \times C_i$ is a pull-back from the degree $(2g-2)$ line bundle on $C$, we have the following formula: $$SW( \tilde{P}) = SW_c ( \tilde{P}|_{X}),$$ where $$c + 8 - 8g = \langle c_1(\mathcal{L})^2, [\hat{X}] \rangle.$$ 
	
\end{prop}
 
 \begin{proof} Let $X_i$ be the manifold obtained from $X$ by filling $\left\{ \infty \right\} \times N_i$ part. Let $P$ be a $Spin^c$-structure on $X$ whose determinant line bundle on $\left\{0 \right\} \times C_i$ is the pull-back of degree $(2g-2)$ line bundle on $C_i$. First, we remark that  there is a naturally extended $Spin^c$-structure $P'$ on $X_i$. We use Theorem \ref{thm2} for $X$ and $D^2 \times C_1$. Then $$SW_c (P') = SW_{c_1} (P)$$ with the property that $c_1 + 4-4g = c$. Then the statement is followed by Corollary 9.9 of \cite{km1}.
 
\end{proof}
 With Proposition \ref{prop9}, we verify the relationship between $\bar{M}$ and $\hat{M}$. We are ready to show the Main Theorem. 
\begin{proof}[Proof of Main Theorem.] 
From Proposition \ref{prop9} and Theorem \ref{thm211}, the main theorem follows naturally. 
\end{proof}

\subsection{The gluing formula along multiple boundaries whose type is $\mathbf{S^1 \times C}$.}\label{sub32}
 In this subsection, we assume that $X_1, X_2$ are compact, oriented, smooth 4-manifolds. Suppose that there are connected, oriented disjoint surfaces $\Sigma_1, \cdots, \Sigma_l \hookrightarrow X_1, X_2$ whose genus are at least 2. Suppose that the manifolds $X_1', X_2'$ are obtained from $X_1, X_2$ by removing the neighborhoods of surfaces $D^2 \times \Sigma_i$ for $i=1,2, \cdots l.$ Then, $\partial X_1' = \partial X_2' = S^1 \times \Sigma_1 \sqcup \cdots \sqcup S^1 \times \Sigma_l.$ We will glue the boundaries of $X_1', X_2'$ along the natural diffeomorphisms. Then we call the resulting manifold X. Moreover, we assume that $b_2^{+} (X_1), b_2^{+} (X_2), b_2^{+} (X) > 1$.
\begin{thm}\label{multiple_gluing}We start with the characteristic cohomology class $k \in H^2 (X, \mathbb{Z})$ satisfying that $s|_{S^1 \times \Sigma_i} = p^{*} k^i$ where $k^i \in H^2 (\Sigma_i, \mathbb{Z})$ satisfies $\langle k^i, [\Sigma_i] \rangle = 2g(\Sigma_i) - 2$ and $p:S^1 \times \Sigma_i \to \Sigma_i$ is a natural projection for $i=1,2, \cdots, l$. Let $k_{X_i'} \in H^2 (X_i', \mathbb{Z})$ be the restriction of $k$ on $X_i'$ for $i=1,2.$ Let $\mathcal{K} (k)$ be a set of all characteristic classes $s \in H^2 (X, \mathbb{Z})$ such that $s|_{X_1'} = k_{X_1'}, s|_{X_2'} = k_{X_2'}.$ Moreover, we define $\mathcal{K}_{X_i} (k)$ as a set of all characteristic classes $s \in H^2( X_i, \mathbb{Z})$ such that $s|_{X_i'} = k_{X_i'}$ for $i=1,2.$ With the appropriate choices of orientations,
\begin{equation}
(-1)^{\star}\sum_{s \in \mathcal{K}(k)} SW_{X} (s) = \sum SW_{X_1} (s_1) SW_{X_2} (s_2)
\end{equation}
where the right hand side sums over $(s_1, s_2) \in \mathcal{K}_{X_1} (k) \times \mathcal{K}_{X_2} (k)$ satisfying that $$s_1^2 + s_2^2 = s^2 - \sum_{i=1}^l (8g(\Sigma_i)-8).$$
\end{thm}
\begin{proof} 
 It is easily proved by applying Theorem \ref{thm21} and \cite[Theorem 3.1.]{km1} repeatedly. We put the orientation term on the left hand side, which is different from the convention in the original paper \cite{km1} for the convenience in Section \ref{sub35}.
\end{proof}
\begin{rmk} As an analogue of \cite[Remark 3.2.]{km1}, two different elements in $\mathcal{K} (k), \mathcal{K}_{X_1} (k),$ $\mathcal{K}_{X_2} (k)$ differ by linear combinations of $[\Sigma_i]^*$, which is a cohomology class which is dual to the second homology class $[\Sigma_i]$, with integer coefficients.
\end{rmk}

\section{Application}
 Mccarthy and Wolfson defined an operation between two symplectic $4$-manifolds, which is called a {\it symplectic normal connect sum} \cite{mcc}. The symplectic normal connect sum is a construction of a new symplectic $4$ manifold from two symplectic manifolds $M_1, M_2.$ Let $\Sigma_1, \Sigma_2$ be symplectic submanifolds embedded in $M_1,M_2$ respectively. Suppose that $\Sigma_1$ has self-intersection number $n \ge 0$ and $\Sigma_2$ has self-intersection number $-n \le 0.$ Let $N_1 (\Sigma_i), N_2(\Sigma_i)$ be the tubular neighborhoods of $\Sigma_i$ inside $M_i$ for $i=1,2.$ Suppose that $N_1 (\Sigma_i)$ is contained in the interior of $N_2(\Sigma_i)$. Let $W_i \subset M_i$ be the complements of the interior of $N_1(\Sigma_i)$ inside $N_2(\Sigma_i)$. Let $f : \Sigma_1 \to \Sigma_2$ be a diffeomorphism. Then there exists an orientation preserving diffeomorphism $\bar{f} : W_1 \to W_2$ induced from $f$ such that $\bar{f}(\partial N_2(\Sigma_1)) = \partial N_1 (\Sigma_2)$. We glue $M_1, M_2$ along $\bar{f}$. The resulting manifold is defined to be the symplectic normal connect sum, denoted by $ M_1 \#_{\bar{f}} M_2.$ It is shown that $X$ is symplectic in \cite{mcc}.
 
 We will examine whether the converse is true in restricted cases. We consider two simple types of $4$-manifolds. Suppose that $M = S^1 \times Y$ for a compact, oriented and connected $3$-manifold and that $\Sigma_1 \subset Y \subset S^1 \times Y$ is an incompressible oriented surface with genus  $g \ge 2$. Let $\Sigma_2$ be a surface homeomorphic to $\Sigma_1$. Suppose that $X$ is a $\Sigma_2$-bundle over an oriented surface $B$ with positive genus. The self-intersection number of $\Sigma_1$ is zero since $\Sigma_1 \subset Y \subset S^1 \times Y$. The self-intersection number of $\Sigma_2$ is also zero since its tubular neighborhood is a product from the definition of the fiber bundle. Therefore, we can construct $X^Y$ which is a normal connect sum of $(M, \Sigma_1)$ and $(X, \Sigma_2)$.

\begin{thm}\label{mainapp} When $b_1 (Y) = 1,$
$X^Y$ has a symplectic form $\omega$ and its canonical structure $K_{\omega}$ such that $\langle K_{\omega}, [\Sigma] \rangle = 2 g(\Sigma) - 2 $ if and only if Y is a surface bundle over the circle.
\end{thm}
  We note that when $g(\Sigma_1) = g(\Sigma_2) = 1,$ the result is proved in \cite{ni}. In this section, we will prove Theorem \ref{mainapp}. 
 
 \subsection{Organization} For fibered three manifolds $Y$, Theorem \ref{mainapp} can be easily shown. We will focus on non-fibered manifolds. In subsection \ref{sub33}, we will construct the covering space of $X^Y$. We first construct covering spaces of $X$ and $M$ respectively and then glue multiple copies of those covering spaces of $X$ and $M$ in a specific way to get a covering space $\tilde{X^Y}$ over $X^Y$. In subsection \ref{sub34}, we introduce the main ingredients to compute the Seiberg-Witten invariant of $\tilde{X^Y}$. If we assume that $X^Y$ has a symplectic structure, then the constructed covering space $\tilde{X^Y}$ also has a symplectic structure. However, we will show that $\tilde{X^Y}$ cannot have a symplectic structure due to the obstruction from the Seiberg-Witten invariants in subsection \ref{sub35}. This completes the proof of Theorem \ref{mainapp}.

\subsection{The construction of $\tilde{X^Y}$.}\label{sub33}
 We will first prove that when the surface bundle over the surface is given, arbitrary covering space over the fiber can be extended to a covering space over the total space.
\begin{lem}\label{thm1}
 Let $\Sigma$ be a connected and orientable surface with genus $g$ which is more than 1 and $\tilde{\Sigma}$ be a connected, orientable surface with genus $ng - n + 1$ for a positive integer $n$. Suppose that a finite $n$-sheeted normal covering $\rho : \tilde{\Sigma} \longrightarrow \Sigma$ is given. Let $X$ be a $\Sigma$-bundle over $B$. Then, there exists a $\tsigma$-bundle $\tilde{X}$ over an oriented surface $\tilde{B}$ such that there exists a covering $\tilde{\rho}: \tilde{X} \longrightarrow X$ satisfying that the restriction of $\tilde{\rho}$ on the fiber $\tsigma$ is isomorphic to $\rho$. 
\end{lem}

\begin{proof}
  We will fix a basepoint $x \in B \subset X$. Let $\Sigma$ be a fiber of $x \in X$. Henceforth, $\Gamma := \pi_1 (\Sigma, x)$. The surface bundle $X$ corresponds to a short exact sequence of fundamental groups \cite[page 51]{got}: $$ 1 \longrightarrow \Gamma = \pi_1(\Sigma,x) \longrightarrow \pi = \pi_1 (X,x) \xrightarrow{q} F = \pi_1(B,x) \longrightarrow 1.$$ Note that $q$ is a quotient map from $\pi_1 (X,x)$ to $\pi_1 (B,x)$. Let $A$ denote the normal subgroup $\rho_{*} (\pi_1 (\tsigma, \tilde{x}))$ of $\Gamma$ for a fixed $\tilde{x} \in \rho^{-1} (x).$ Let $N$ be the normalizer of $A$ in $\pi$ i.e. $N = \left\{ g \in \pi : g A g^{-1} = A \right\}.$ $A \unlhd N$ and $\Gamma \unlhd N$ since $\Gamma \unlhd \pi$.

  We will first show that $N/\Gamma$ is a surface group. If we show that $N/\Gamma$ is a finite index subgroup of $\pi/\Gamma,$ then it must be a surface subgroup. Let $\Lambda$ be the set of subgroups of $\Gamma$ whose index is equal to $|\Gamma / A|.$ Clearly, $A \in \Lambda$ and $\Lambda$ is a finite set since $\Gamma$ is finitely generated.

  $\pi$ has an action on $\Lambda$: for $g \in \pi$ and $H \in \Lambda$, $$g \cdot H = g H g^{-1} \in \Lambda.$$ This action naturally defines the map $\Phi : \pi \to Perm(\Lambda),$ where $Perm(\Lambda)$ is a permutation group of $\Lambda$ which is finite.  Obviously, $N = \left\{ g \in \pi | g \cdot A = A \right\}.$ $\pi / N$ is a finite set since $\ker \Phi \le N$ and $\pi/\ker \Phi = Perm(\Lambda)$ is finite. Therefore, $N/ \Gamma$ is a surface group since it is a finite index subgroup of the surface group $\pi /\Gamma.$
 
 There is a short exact sequence 
\begin{equation}\label{eqn1}
1 \longrightarrow \Gamma/A \hookrightarrow N/A \longrightarrow N / \Gamma \longrightarrow 1.
\end{equation}
 
\begin{lem}\label{lem1} For a finite group $G$ and a surface group $S$, suppose that there is a group extension $H$ satisfying $$1 \longrightarrow G \hookrightarrow H \longrightarrow S \longrightarrow 1.$$ Then, $H$ always contains a surface subgroup, which is also isomorphic to a proper subgroup of $S$.  
\end{lem}

\begin{proof} Let $C$ be the centralizer of $G$ in $H$. i.e. $C = \left\{ h \in H : hx = xh \text{ for all } x \in G \right\}.$ We show that $C$ is a finite index subgroup of $H$. 
\begin{itemize}
\item $H$ has an action on itself defined by conjugation: 
\begin{center}
$g \cdot x = gxg^{-1}$ for all $g,x \in H$.
\end{center}
\item Since $G \unlhd H$, $g \cdot G = G$. Therefore, the action gives the homomorphism $\phi$ from $H$ to $Aut(G)$, the automorphism group of $G$. Since $Aut(G)$ is a finite group, $\ker \phi$ becomes a finite-index subgroup of $H$. Moreover, $\ker \phi = C.$
\end{itemize} 
 We have the following short exact sequence:
\begin{equation}\label{eqn2} 1 \longrightarrow C \cap G \longrightarrow C \longrightarrow S' \longrightarrow 1, \end{equation} where $S' = C/C\cap G.$
 If we define a map $ \Phi : S' \longrightarrow S$ by the natural inclusion, then this map is well-defined. Moreover, $\Phi$ is injective. Therefore, $S' = C / C \cap G$ is isomorphic to a finite-index subgroup of $S = H/G$.

 It is easily seen that $C \cap G$ is contained in the center of $C$. Therefore, a short exact sequence \ref{eqn2} represents a central extension. This central extension corresponds to an element inside $H^2 (S', C\cap G)$. Note that if the corresponding element is zero, then the short exact sequence is splittable.  $H^2 (S', C\cap G)$ is finite and has only torsion elements since $C \cap G$ is finite.

 We examine an arbitrary finite-index subgroup $T$ of $S'$. Let $n$ be $[T:S'].$ Assume that $a \in H^2 (S', C \cap G)$ be the corresponding element to the short exact sequence. Then, $na \in H^2 (T, C \cap G)$ be the corresponding element of the short exact sequence 
\begin{equation}\label{eqn34}1 \longrightarrow C \cap G \longrightarrow q^{-1} (T) \longrightarrow T \longrightarrow 1,
\end{equation} where $q^{-1} (T)$ is a subgroup of $C.$

 Since $C \cap G$ is a finite group, there exists a positive integer $m$ such that $ma = 0.$ If we use this $m$ to pick the subgroup $T$, then the short exact sequence \ref{eqn34} becomes splittable. Hence, the surface group $T$ becomes a subgroup of $q^{-1}(T) \le C \le H$. Moreover, $T \le S' \le S$. Therefore, the statement is proved. 
\end{proof}
 
 We prove Lemma \ref{thm1} based on Lemma \ref{lem1}. We apply Lemma \ref{lem1} to the short exact sequence \ref{eqn1}. From the lemma, there exists a surface subgroup $H$ of $N/A$, which is also a subgroup of $\pi/\Gamma.$ Let $H$ be isomorphic to the fundamental group of $\Sigma_h$, which is an oriented surface with genus $h \ge 1$. With $H$, the following diagram $(\star)$ commutes.
\begin{center}
$(\star)$ \hspace{5mm}
\begin{tikzcd}
  1 \arrow[r] & A \arrow[r] \arrow[d, hookrightarrow]  &q^{-1} (H) \arrow[r, "q"] \arrow[d, hookrightarrow] &H \arrow[r] \arrow[d, hookrightarrow]&1  \\
  1 \arrow[r] &\Gamma \arrow[r] &\pi \arrow[r] &\pi/\Gamma \arrow[r] &1
\end{tikzcd}
\end{center}

  We will construct the covering $\tilde{\rho}.$ First, we consider a covering space $B'$ over $B$ corresponding to $H$, that is a subgroup of $\pi/\Gamma = \pi_1(B,x)$. Then, there is a pull-back $\Sigma$-bundle $X'$ over $B'$ and the covering map $X' \to X.$
\begin{center}
\begin{tikzcd}
  X' \arrow[r] \arrow[d] & X \arrow [d] \\
  B' \arrow[r] & B 
\end{tikzcd}
\end{center}
  
  Next, we will construct a covering space $\tilde{\rho} : \tilde{X} \to X'$ corresponding to $q^{-1} (H)$. Finally, this $(\tilde{X}, \tilde{\rho})$ will be a covering space over $X$. From the $(\star)$ diagram, the restriction of $\tilde{\rho}$ on the fiber is equal to $\rho.$  Therefore, Lemma \ref{thm1} is proved.
\end{proof}
  
  We have all of the ingredients necessary to construct the covering space of $X^Y$ by gluing several copies of $\tilde{X}$ and $S^1 \times \tilde{Y}$ together.

\begin{lem}\label{covering}
 There exists a finite cover $\tilde{X^Y}$ of $X^Y$ satisfying $b^{+}_2 (\tilde{X^Y}) >1.$
\end{lem}
\begin{proof}
 The following construction is an analogue of \cite{ni}. We start with a finite normal covering space $\tilde{Y}$ over $Y$ and a covering map $p : \tilde{Y} \to Y.$ Then, there is a covering spaces $p : S^1 \times \tilde{Y} \to S^1 \times Y$ defined by $$(z,y) \to (z^m, p(y))$$ for a fixed positive integer $m>1.$ Henceforth, $\tilde{M}, M$ denote $S^1 \times \tilde{Y}, S^1 \times Y$ respectively. Let $l$ be the number of components of $p^{-1} (\Sigma).$ $l> 1$ since $l \ge m.$ Since $p$ is normal, all the components of $p^{-1} (\Sigma)$ are homeomorphic. Let $\tilde{\Sigma}$ be the oriented surface homeomorphic to the component of $p^{-1} (\Sigma).$ Let $p|_{\tsigma} : \tilde{\Sigma} \longrightarrow \Sigma$ be a $n-$fold covering. We apply Lemma \ref{thm1} to $p|_{\tsigma}$. Then we get a covering space $\tilde{p} : \tilde{X} \to X$. Let $r$ be the number of components of $\tilde{p}^{-1} (\Sigma)$. Note that $r$ is divisible by $[\pi_1(B):H] >1$ with the notation in Lemma \ref{thm1}, hence $r>1.$ 
 
 We take $r$ copies of $S^1 \times \tilde{Y}$ and $l$ copies of $\tilde{X}$. For each copy of $S^1 \times \tilde{Y}$, there are $l$ copies of $\tilde{\Sigma}$. On the other hand, for each copy of $\tilde{X}$, there are $r$ copies of $\tilde{\Sigma}$. We correspond each $\tilde{\Sigma}$ in  one copy of $S^1 \times \tilde{Y}$ to $\tilde{\Sigma} \subset \tilde{p}^{-1} (\Sigma)$ for each copy of $\tilde{X}.$ For each pair of $\tilde{\Sigma}$ in $\tilde{M}$ and $\tilde{\Sigma}$ in $\tilde{X}$, we remove the neighborhoods of $\tilde{\Sigma}$  and glue along their boundaries $S^1 \times \tilde{\Sigma}$. The gluing maps are all given by a lifting of the gluing map between $S^1 \times Y$ and $X$. After gluing all, this forms a covering space $X^Y_{r,l}$ over $X^Y$ which retracts onto $K_{r,l}$, a complete bipartite graph. Remark that $X^Y_{r,l}$ is a $nlr$-fold covering of $X^Y$. We will show that $X^Y_{r,l}$ satisfies the condition. 
  
  We show that $b_2^{+} (X^K_{r,l}) >1.$ From the definitions of Euler characteristic and signature, $$\chi(\tilde{X^Y}) = (2b^{+}_2 (\tilde{X^Y}) - 2 b_1 (\tilde{X^Y}) + 2 )- \sigma(\tilde{X^Y}).$$ It is known that $\sigma(X) = (2-2g)(2-2b)$ where $g(\Sigma) = g$ and $g(B) = b$ and $\chi(S^1 \times \tilde{\Sigma}) = 0.$ Therefore, it is easily shown that $\chi(X^Y) = (2-2g)(-2b)$. From the Novikov additivity property and $\sigma(S^1 \times \tsigma) = 0$, $\sigma(X^Y) = \sigma(X)$. We cannot specify the exact value of $\sigma(X)$; however, it has a bound $$\frac{-(b-1)(g-1)}{2} \le \sigma (X) \le \frac{(b-1)(g-1)}{2},$$ which is proved in \cite{kot}.  Therefore,
\begin{align*}
b^{+}_2 (\tilde{X^Y}) &= b_1 (\tilde{X^Y}) - 1 + \frac{nlr}{2} ( \chi ( X^Y) + \sigma(X^Y)) \\ 
&= (l-1)(r-1) - 1 + \frac{nlr}{2}(\chi(X^Y) + \sigma(X^Y)) \\
&> (l-1)(r-1)-1 + \frac{nlr}{2} ( 2b(2g-2) - \frac{(b-1)(g-1)}{2}) > 1
\end{align*}
\end{proof}

\subsection{The main ingredients in the computation of Seiberg-Witten invariant of $\tilde{X^Y}$.}\label{sub34} In this subsection, we will state theorems which are necessary to compute the Seiberg-Witten invariant of $\tilde{X^Y}.$ On one hand, we will discuss that the Seiberg-Witten invariant of $4$-manifolds $S^1 \times Y$, for a compact, oriented and connected $3$-manifold $Y$ with $b_1 (Y) >0$, is related to the Alexander polynomial of the $3$-manifold $Y$ \cite{meng}. On the other hand, we will discuss the Seiberg-Witten invariant of the canonical $Spin^c$-structure of symplectic manifolds. Before the statement, we need two new definitions. 
 First, let $H(X) = H^2(X)/\text{Tors}$ be a non-torsion part of the second cohomology $H^2 (X)$ for a closed manifold $X$. Second, we have the Seiberg-Witten invariants $SW : H^2 (X) \to \mathbb{Z}$ for $4$-manifolds $X$. We have the natural quotient map $q:H^2 (X) \to H(X).$ Then we define $$\underline{SW}_X = \sum_{z \in H^2 (X)} SW_X (z) q(z) \in \mathbb{Z}[H(X)].$$ 
 
\begin{thm}\label{thm37} \cite{meng} Let $N$ be a compact, oriented and connected $3$-manifold with $b_1(N)>0$ such that the boundary of $N$ is empty or a disjoint union of tori. Let $p : S^1 \times N \to N$ be a natural projection and $p_{*} : H(N) \to H(S^1 \times N)$ be the induced homomorphism by $p$. Obviously, there exists a natural homomorphism $\Phi_2 : \mathbb{Z}[H(N)] \to \mathbb{Z}[H(S^1 \times N)]$ induced by $2p_{*}.$ When $b_1(N)=1$, $H(N) \cong H_1(N)/\text{Tors} \cong \mathbb{Z}$. Let $t$ be a generator of $H(N)$. Then, there exists an element $\xi \in \pm p_{*} (H(N))$ such that 
\begin{equation*}
\underline{SW}_{S^1 \times N} = \begin{cases*}
      \xi \Phi_2 (\Delta_N) & \text{if $b_1(N) >1$}  \\
      \xi \Phi_2 ((1-t)^{|\partial N|-2} \Delta_N) & \text{if $b_1(N) =1$}
    \end{cases*}
\end{equation*}
\end{thm}

If we apply Theorem \ref{thm37} to $\tilde{Y}$, then
\begin{equation}\label{underline_eq}
\underline{SW}_{S^1 \times \tilde{Y}} = \begin{cases*}
      \xi \Phi_2 (\Delta_{\tilde{Y}}) & \text{if $b_1(\tilde{Y}) >1$}  \\
      \xi \Phi_2 ((1-t)^{-2} \Delta_{\tilde{Y}}) & \text{if $b_1(\tilde{Y}) =1$}.
      \end{cases*}
\end{equation}

 To make the Seiberg-Witten invariant of $S^1 \times \tilde{Y}$ easy to compute, we will choose the covering $\tilde{Y}$ such that $\Delta_{\tilde{Y}}$ is almost trivial. The following three theorems make it easy to deal with the Alexander polynomials of covering spaces. 

\begin{prop}\label{prop38}\cite[Proposition 3.6]{fv} Let $N$ be a $3$-manifold and let $\alpha : \pi_1 (N) \to G$ be an epimorphism onto a finite group. Let $H_G$ be $H(N_G)$ and $H$ be $H(N)$. Let $\pi_{*} : H_G \to H$ be the induced map  Then the twisted Alexander polynomials of $N$ and the ordinary Alexander polynomial of $N_G$ satisfy
the following relations:\\

If $b_1(N_G)$ > 1, then 
\begin{equation*} 
\Delta^{\alpha}_N = \begin{cases} \pi_* ( \Delta_{N_G}) &\text{if } b_1 (N) >1 \\ (a-1)^2 \pi_* (\Delta_{N_G}) &\text{if } b_1 (N) = 1, im \pi_* = \langle a \rangle \end{cases} 
\end{equation*}

If $b_1 (N_G) = 1$, then $b_1 (N) = 1$ and $$\Delta^{\alpha}_N = \pi_{*} (\Delta_{N_G}).$$

\end{prop}
 We call $\phi \in H^1 (N)$ {\it fibered} if $\phi$ is dual to a fiber of a fibration $N$ over $S^1.$ Friedl and Vidussi proved the vanishing theorem of the twisted Alexander polynomial. 
\begin{thm}\label{thm39}\cite[Theorem 2.3]{fv2} Let $N$ be a compact, orientable, connected 3-manifold with (possibly empty) boundary consisting of tori. If $\phi \in H^1(N)$ is not fibered, then there exsits an epimorphism $\alpha : \pi_1 (N) \to G$ onto a finite group $G$ such that $$\Delta^{\alpha}_{N, \phi} = 0.$$
\end{thm}

\begin{prop}\label{thm36} When $b_1 (Y) = 1$ and $Y$ is not fibered, there exists a normal finitely sheeted covering space $\pi : \tilde{Y} \to Y$ such that $\pi_{*} (\Delta_{\tilde{Y}}) = 0,$ where $\pi_{*} : \mathbb{Z}[H(\tilde{Y})] \to \mathbb{Z}[H(Y)]$ is an induced homomorphism by the covering map $\pi$.
\end{prop}

\begin{proof}
 This statement is straightforward based on Proposition \ref{prop38} and Theorem \ref{thm39}. We can pick an epimorphism $\alpha : \pi_1(Y) \to G$ onto a finite group $G$ such that the covering space $\tilde{Y}$ over $Y$ corresponding to $\ker \alpha$ satisfies that $\pi_{*} ( \Delta_{\tilde{Y}}) = 0 \in \mathbb{Z}[H(Y)].$
\end{proof}

\begin{rmk}\label{rmk311} In case of $b_1 (Y) > 1$, we cannot have the result in Proposition \ref{thm36}, however we have a slightly weaker result. 
From \cite[Equation (5)]{fv2},
\begin{equation*} 
\Delta^{\alpha}_{N,\phi} = \begin{cases} (t^{\text{div}\phi_G} -1 )^2 \phi ( \Delta^{\alpha}_{N}) &\text{if } b_1 (N) >1 \\  \phi (\Delta^{\alpha}_{N}) &\text{if } b_1 (N) = 1. \end{cases} 
\end{equation*}
 Theorem \ref{thm39} asserts that we can pick $\phi \in H^1(N)$ such that $\Delta^{\alpha}_{N, \phi} = 0$. For both cases, $\phi ( \Delta^{\alpha}_N) = 0.$ Combining with Proposition \ref{prop38}, $\phi ( \pi_{*} ( \Delta_{N_G})) = 0.$ We will visit this remark later and explain the case $b_1 (Y) > 1.$ 
\end{rmk}

 The Seiberg-Witten invariants of symplectic manifolds are well-known for the canonical $Spin^c$-structure. 
\begin{thm}[\cite{tau}, \cite{tau2}]\label{thm38} Let $X$ be a symplectic four manifold with symplectic form $\omega$. Let $K_X \in H^2 (X, \mathbb{Z})$ be the canonical class of the symplectic structure. If $b^{+}_2 (X) >1$, then $SW_X ( K_X) = \pm 1.$ Moreover, for each $Spin^c$-structure $\mathfrak{s}$ such that $SW_X (s) \ne 0$, $$|c_1 (\mathfrak{s}) \smile \omega| \le c_1 (\mathfrak{t}) \smile \omega$$ and the equality holds if and only if $\mathfrak{s} = \mathfrak{t}$ or $\bar{\mathfrak{t}}$.
\end{thm}

\subsection{Proof of Theorem \ref{mainapp}}\label{sub35}
 First, we consider a case in which Y is a surface bundle over $S^1$.  If Y is a surface bundle over $S^1$, then $S^1 \times Y$ has a symplectic structure according to Friedl and Vidussi \cite{fv}. $X^Y$ is a normal connected sum of two symplectic $4$-manifolds $S^1 \times Y$ and $X$ along symplectically embedded surfaces $\Sigma$.  As a result, $(X^Y, \omega)$ is symplectic. Moreover, based on adjuction formula, the canonical symplectic form $K_{\omega}$ satisfies that $\langle K_{\omega}, [\Sigma] \rangle = 2g - 2$ since $[\Sigma]^2 = 0$ in $X^Y.$
 
 Before moving to the other case, we define the following notations.
\begin{itemize}
\item Let $Z$ be $S^1 \times Y \setminus nbd(\Sigma)$, where $nbd(\Sigma)$ is an neighborhood of $\Sigma$, homeomorphic to $D^2 \times \Sigma$. 
\item Let $X'$ be $X \setminus nbd(\Sigma)$. 
\item $X^Y$ is obtained from gluing $Z$ and $X'$ along their boundaries.

\end{itemize}

 Second, suppose that Y is not a surface bundle over $S^1$ and that $X^Y$ has a symplectic structure. Now, by Lemma \ref{covering}, there exists a covering map $p : \tilde{X^Y} \to X^Y$ satisfying the conditions stated in the lemma. If we consider the submanifold $p^{-1} (Z) \subset \tilde{X^Y},$ then $p^{-1} (Z)$ has $r$ components. We pick one component among them and call it $M_1'.$ The boundary of $M_1'$ consists of $l$ copies of $S^1 \times \tilde{\Sigma}.$ We fill the boundaries of $M_1'$ by $D^2 \times \tilde{\Sigma}$ trivially. Then, we get a closed manifold $M_1$ which is homeomorphic to $S^1 \times \tilde{Y}.$ Moreover, let $M_2' = \tilde{X^Y} \setminus M_1'.$ Then we fill the boundaries of $M_2'$ by $D^2 \times \Sigma$ trivially. We call the resulting manifolds $M_2.$ Remark that $M_1$ is homeomorphic to $S^1 \times \tilde{Y}.$ Conversely, if we do a normal connected sum between $M_1, M_2$ along $l$ copies of $\tilde{\Sigma}$ repeatedly, then the resulting manifold becomes $\tilde{X^Y}$. The first procedure involves gluing two separated $4$-manifolds, while the next involves self-gluing $l-1$ times. We use the two gluing theorems for $M_1, M_2$ and $\tilde{X^Y}.$ We will use Theorem \ref{multiple_gluing} when a $Spin^c$-structure $\tilde{P}$ is given on $\tilde{X^Y}.$

 Henceforth, we will say that $K \in H^2 (S^1 \times C, \mathbb{Z})$ satisfies the pull-back condition for an orientable surface $C$ when $K|_{S^1 \times C} \in H^2 (S^1 \times C, \mathbb{Z}) = \rho^{*} (k_0)$ where $\rho : S^1 \times C \to C$ is a natural projection and $k_0 \in H^2(C, \mathbb{Z})$ satisfies that $\langle k_0,C \rangle = 2g(C) - 2$. From now on, $H^2(-)$ denotes a cohomology group with integer coefficients if not specified.
 
 Suppose that $\omega$ is a symplectic $2$-form on $X^Y$. Let $K_{\omega} \in H^2 (X^Y, \mathbb{Z})$ be the canonical class of the symplectic structure $\omega$ on $X^Y.$ Since $K_{\omega}$ is not torsion from the assumption $\langle K_{\omega} , [\Sigma] \rangle \ne 0,$ we can perturb $\omega \in H^2 (X^Y, \mathbb{R})$ to be a rational cohomology class and then scale $\omega$ properly so that $[\omega] \in H^2(X^Y, \mathbb{Z}).$ We will show that $K_{\omega}$ satisfies the pull-back condition. Based on adjunction inequality, for any closed curve $\alpha$ which is homologically nontrivial in $C$, $T^2 = S^1 \times \alpha \in S^1 \times C$, $\langle K_{\omega}, S^1 \times \alpha \rangle = 0.$
 Moreover, from the assumption $$\langle K_{\omega}, [\Sigma] \rangle = 2g(\Sigma) -2.$$ Consequently, $K_{\omega}$ satisfies the pull-back condition.
 
  Let $\Omega$ be the pull back $2$-form of $\omega$ on $\tilde{X^Y}.$ Let $K_{\Omega} \in H^2(\tilde{X^Y}, \mathbb{Z})$ be the canonical class of the symplectic structure $\Omega$ on $\tilde{X^Y}$. $K_{\Omega} = p^{*} ( K_{\omega} ).$ Likewise, $K_{\Omega} \in H^2(S^1 \times \tsigma)$ fulfills the pull-back condition on every component. In conclusion, we can use Theorem \ref{mainapp} for this canonical structure $K_{\Omega}.$ $i_1^{*}, i_2^{*}$ denote the natural maps $H^2(\tilde{X^Y}, \mathbb{Z}) \to H^2(M_1')$ and $H^2(\tilde{X^Y}, \mathbb{Z}) \to H^2(M_2')$ induced by inclusion maps $M_1', M_2' \hookrightarrow \tilde{X^Y}.$ Henceforth, $K_{M_1'}, K_{M_2'}$ denote $i_1^{*} (K_{\Omega}), i_2^{*} (K_{\Omega})$ respectively. We pick $k \in H^2 (\tilde{X^Y}, \mathbb{Z})$ which also satisfies the condition that $k|_{S^1 \times \tilde{\Sigma}} = \rho^{*}(k_0)$. Then, the restriction $(k - K_{\Omega})$ on $S^1 \times \tilde{\Sigma}$ becomes zero. Moreover, let $k_{M_1'} = i_1^{*} ( k), k_{M_2'} = i_2^{*} ( k).$ Then, $k_{M_1'} - K_{M_1'}$ and $k_{M_2'} - K_{M_2'}$ vanish on the boundary.

 We focus on the element $(k - K_{\Omega}) \smile [\Omega]$ in $H^4(\tilde{X^Y}, \mathbb{Z}) \cong \mathbb{Z}.$ We can decompose the integer corresponding to $(k - K_{\Omega}) \smile [\Omega]$ in the following way. First, $[F] = PD[(k-K_{\Omega})] \in H_2 (\tilde{X^Y})$, where a surface $F \in C_2 (\tilde{X^Y})$ has no intersection with $S^1 \times \tilde{\Sigma}.$ Next, $F$ can be decomposed into $F_1 + F_2$ for $F_1 \in C_2 (M_1'), F_2 \in C_2 (M_2').$ Then,
\begin{equation}\label{restriction}
 (k - K_{\Omega}) \smile [\Omega] = PD[F_1] \smile i_1^{*}[\Omega] + PD[F_2] \smile i_2^{*}[\Omega].
\end{equation} The equation \ref{restriction} verifies the relationship between $Spin^c$-structures $k \in H^2(\tilde{X^Y}, \mathbb{Z})$ and its restrictions to $M_1', M_2'.$ 
  
  We will add the equation in Theorem \ref{multiple_gluing},  $(-1)^{\star} \sum_{s \in \mathcal{K}(k)} SW_{X} (s) = \sum SW_{X_1} (s_1) SW_{X_2} (s_2)$ for $k \in H^2(\tilde{X^Y})$ satisfying that 
\begin{enumerate}
\item  $k|_{S^1 \times \tilde{\Sigma}} = \rho^{*}(k_0)$ where $\rho : S^1 \times \tilde{\Sigma} \to \tilde{\Sigma}$ and $k_0 \in H^2 (\tilde{\Sigma})$ satisfies that $\langle k_0, [\tilde{\Sigma}] \rangle = 2g(\tilde{\Sigma})-2$.
\item $k \smile [\Omega] = K_{\Omega} \smile [\Omega] \in \mathbb{Z}$
\item $k^2 = K_{\Omega}^2.$

\end{enumerate}  
  
  It is known that only $K_{\Omega}$ satisfies all of the above properties (1)-(3) and $SW \ne 0$. Therefore, the left hand side is exactly equal to $SW(K_\Omega)$ which is $+1$ or $-1$ from the Theorem \ref{thm38}. The right hand side becomes $\displaystyle \sum_{(\hat{z_1}, \hat{z_2}) \in H^2 ( M_1) \times H^2 (M_2)} SW_{M_1} (\hat{z_1}) SW_{M_2} (\hat{z_2})$ for $(\hat{z_1}, \hat{z_2}) \in H^2 ( M_1) \times H^2 (M_2)$ satisfying 
\begin{enumerate}
\item $\hat{z_1}^2 + \hat{z_2}^2 = k^2 - (4g-4)l$
\item There exist $[F_1], [F_2] \in H_2 (M_1'), H_2 (M_2')$ such that \begin{itemize}\item $F_1 \cap \partial M_1' = F_2 \cap \partial M_2' = \phi$ \item $(\hat{z_1}|_{M_1'} - K_{M_1'}) = j_1 (PD[F_1])$ and $(\hat{z_2}|_{M_2'} - K_{M_2'}) = j_2 (PD[F_2])$ where $j_1 : H^2(M_1', \partial M_1') \to H^2(M_1'), j_2 : H^2(M_2', \partial M_2') \to H^2(M_2')$ are natural maps.\end{itemize}
\item $[F_1], [F_2]$ determined in (2) satisfy that 
$PD[F_1] \smile i_1^{*} [\Omega] + PD[F_2] \smile i_2^{*} [\Omega] = 0.$
\end{enumerate}
 Now we will see 
\begin{equation}\label{inner}
\sum_{(\hat{z_1}, \hat{z_2}) \in H^2 ( M_1) \times H^2 (M_2)} SW_{M_1} (\hat{z_1}) SW_{M_2} (\hat{z_2}) = \sum_{\hat{z_2}} SW_{M_2} (\hat{z_2}) (\sum_{\hat{z_1}} SW_{M_1} (\hat{z_1})).
\end{equation}
 We want to prove that the inner sum $\displaystyle \sum SW_{M_1}(\hat{z_1})$ in the right hand side becomes zero. Remark that if $SW_{M_1} (\hat{z_1}) \ne 0$, then $\hat{z_1}^2 = 0$. Therefore, the first condition does not need to be considered in the inner sum of Equation \ref{inner}. Let $Z_m$ be a set of $\hat{z_1} \in H^2 (S^1 \times \tilde{Y})$ such that the corresponding $F_1$ satisfies $PD[F_1] \smile i_1^{*} [\Omega] = m.$ In other words, we will show that $\displaystyle \sum_{\hat{z_1} \in Z_m} SW_{M_1}(\hat{z_1}) =0$ for all $m$. 

 We have the formula for the Seiberg Witten invariants. For $h \in H^2(S^1 \times \tilde{Y})$, $[h]$ denotes the quotient element in $H(S^1 \times \tilde{Y}) \cong H^2(S^1 \times \tilde{Y}) / Tors.$ If two elements $x,y \in H^2 (S^1 \times \tilde{Y})$ satisfy that $[x]=[y]$, then $x \in Z_m$ is equivaltent to $y \in Z_m$ since the result of cup products does not depend on the torsion part.   
   Let $$\Delta_{\tilde{Y}} = \sum_{h \in H(\tilde{Y})} g_h \cdot h$$ for $g_h \in \mathbb{Z}.$ We have the natural projection $p : S^1 \times \tilde{Y} \to \tilde{Y}$ and the induced map $p^{*} : H(\tilde{Y}) \to H(S^1 \times \tilde{Y}) \cong H(\tilde{Y}) \oplus ( \mathbb{Z} \otimes H^1(\tilde{Y}) / Tors).$ The image of $p^{*}$ is included in the first summand. 
 
 First, suppose that $b_1 (\tilde{Y}) >1$. Then from Theorem \ref{underline_eq}, $$\underline{SW}_{S^1 \times \tilde{Y}} = \xi \Phi_2 (\Delta_{\tilde{Y}}).$$ From Kunneth formula, $H^2(S^1 \times \tilde{Y}) = H^2(\tilde{Y}) \oplus (H^1(S^1) \otimes H^1(\tilde{Y}))$. Since $\xi \in \pm p^{*}(H(\tilde{Y}))$, $\xi$ supports on the first summand of $H (S^1 \times \tilde{Y}).$ Conclusively,
 $$\underline{SW}_{M_1} = \xi \Phi_2 (\Delta_{\tilde{Y}})$$ 
$$\Longrightarrow \sum_{h \in H (M_1)} SW_{M_1} (h) [h] = \sum_{h \in H(S^1 \times \tilde{Y})} g_h \cdot (2h+\xi)$$
Therefore, we can summarize the equality:
\begin{enumerate}
\item $\displaystyle \sum_{\substack{h \in H^2(S^1 \times \tilde{Y}),\\ [h] = 2l+\xi \in H(S^1 \times \tilde{Y})}} SW(h) = g_l$ for $ l \in H (\tilde{Y}) \subset H(S^1 \times \tilde{Y}).$
\item Otherwise, $\displaystyle \sum_{[h] = k} SW(h) = 0$ where $k$ is not represented by $2l+\xi.$
\end{enumerate}

 We define $\phi : H^2 (S^1 \times \tilde{Y}) \to H(Y)$ to be the composition of trivial quotient maps $H^2(S^1 \times \tilde{Y}) \to H(S^1 \times \tilde{Y}) \to H(\tilde{Y})$ and $\pi_{*} : H(\tilde{Y}) \to H(Y)$. Therefore,
\begin{equation}\label{cov}
\sum_{h \in \pi_{*}^{-1}(x)} g_h = 0
\end{equation}
for each $x \in H(Y).$
 This is equivalent with 
\begin{equation}\label{cov2}
\sum_{\substack{h \in H^2(S^1 \times \tilde{Y}) \\ \phi (h) = x} } SW_{S^1 \times \tilde{Y}} (h) = 0
\end{equation}
for each $x \in H(Y).$ 

\begin{lem}\label{lem314} If $\hat{z}, \hat{w} \in  H^2 (\tilde{Y}) \subset H^2(S^1 \times \tilde{Y})$ satisfy that $\phi(\hat{z}) = \phi(\hat{w})$, then $$\hat{z} \in Z_k \iff \hat{w} \in Z_k$$
\end{lem}
\begin{proof} Let $F_z, F_w$ be the $2$-chain corrsponding to $\hat{z}, \hat{w}$ respectively. We observe that the covering $p$ restricted to $M_1'$ is isomorphic to the covering $p_1 : S^1 \times \tilde{Y} \to S^1 \times Y$ restricted to $M_1'.$ Let $p' : M_1' \to p(M_1') = Z.$ Recall that $Z \cong S^1 \times Y \setminus D^2 \times \Sigma.$ The diagram \ref{diagram} commutes. The horizontal maps are covering maps $p, p'$ and the vertical maps are natural inclusion maps. 
\begin{figure}
\begin{tikzcd}
  \tilde{X^Y} \arrow[r, "p"] 
    & X^Y  \\
  M_1' \arrow[r, "p"] \arrow[u, hookrightarrow]
&Z \arrow[u, hookrightarrow]
\end{tikzcd}
\caption{Diagram \ref{diagram}}\label{diagram}
\end{figure}
 Recall that $\Omega = p^{*} [ \omega] $. Therefore, $i_1^{*} [\Omega] = p'^{*} [\omega|_{Z}].$ Moreover, we also have the following commuting diagram \ref{diagram2}.

\begin{figure}
\begin{tikzcd}
  \tilde{M_1} \arrow[r, "p"] 
    & S^1 \times Y  \\
  M_1' \arrow[r, "p"] \arrow[u, hookrightarrow]
&Z \arrow[u, hookrightarrow]
\end{tikzcd}
\caption{Diagram \ref{diagram2}}\label{diagram2}
\end{figure}
 Hence,
\begin{align*}
PD[F_z] \smile i_1^{*} [\Omega] &= \langle [F_z], i_1^{*} [\Omega] \rangle\\
&= \langle [F_z], p'^{*} [\omega|_{Z}] \rangle\\
&= \langle [p'(F_z)], [\omega|_{Z}] \rangle\\
&= PD[p'(F_z)] \smile [\omega|_{Z}] \\
&= \langle p'(F_z), [\omega|_{Z}] \rangle 
\end{align*}

 The second commutative diagram indicates that $p'(F_z) = p'(F_w).$ Therefore, the statement is true. 

\end{proof}
 Finally, based on Lemma \ref{lem314}, the inner sum of Equation \ref{inner} can be decomposed into the sum of Equation \ref{cov2} for some $x$. Therefore, the inner sum of Equation \ref{inner} becomes zero. 
 
 Second, suppose that $b_1(\tilde{Y}) = 1.$ Since $b_1(Y)=1$, $\pi_{*} : H(\tilde{Y}) \to H(Y)$ is a homomorphism from $\mathbb{Z} \to \mathbb{Z}.$ Since this homomorphism is not trivial, $\pi_{*}$ is injective. Therefore, $\pi_{*} (\Delta_{\tilde{Y}}) = 0$ implies that $\Delta_{\tilde{Y}} = 0.$ Therefore, $\displaystyle \sum_{[h] = l} SW(h) = 0$ for all $l \in H(S^1 \times \tilde{Y}).$ Therefore, the inner sum of the right hand side is also zero. 
 
 Therefore, this is a contradiction because the Equation \ref{eqn2} is not true. This implies that $X^Y$ does not have a symplectic structure and hence concludes the proof of Theorem \ref{mainapp}.

\begin{rmk} We add a remark in a more general setting. According to Remark \ref{rmk311}, in case of $b_1 (Y) > 1,$ the corresponding equation to Equation \ref{cov} is for each $m \in \mathbb{Z},$ $$ \sum_{\psi \smile \pi^{*} (h) = m} g_h = 0 $$ for some non-fibered $\psi \in H^1(Y).$ Then, to extend the results to the case of $b_1 (Y)>1,$ Lemma \ref{lem314} should be transformed into: for two $\hat{z}, \hat{w}$ satisfying that $\psi \smile \pi(\hat{z}) = \psi \smile \pi(\hat{w}),$ the statement is true. However, we cannot use the same proof since $2$-form $\omega$ on $X^Y$ does not have a correspondence with $2$-form on $S^1 \times Y.$ 
\end{rmk}

\end{document}